\documentclass[12pt]{article}
\usepackage{amsmath}
\usepackage[dvipsnames,usenames]{color}
\usepackage{amsfonts}
\usepackage{mathrsfs}
\usepackage{amssymb}
\usepackage{amsthm}
\def\h{\hbox}
\def\<{\langle}
\def\>{\rangle}
\numberwithin{equation}{section}
\def\<{\langle}
\def\>{\rangle}

\def\ra{\rightarrow}
\def\p{\partial}
\def\a{\alpha}

\def\B{\mathcal B}
\def\O{\Omega}
\def\l{{\lambda}}

\def\PP{{\mathbb P}}
\def\BB{{\mathbb B}}
\def\F{\mathcal F}

\def \sm{\setminus}
\def\-{\overline}

\def\o{\omega}
\def\ov{\overline}
\def\e{\epsilon}

\def\L{\Lambda}
\def\h{\hbox}
\def\d{\delta}

\def\d{\delta}
\def\d{\delta}
\def\b{\beta}
\def\a{\alpha}

\def\RR{{\Bbb R}}
\def\CC{{\Bbb C}}

\def\NN{{\Bbb N}}

\def\BB{{\Bbb B}}

\def\ov{\overline}

\def\ld{\lambda}
\def\O{\Omega}
\def\o{\omega}

\def\sm{\setminus}
\def\L{\Lambda}

\def\h{\hbox}

\def\wt{\widetilde}
\def\ra{\rightarrow}
\def\p{\partial}

\def\zbar{\ov z}

\def\d{\delta}

\def\a{\alpha}
\def\d{\delta}

\def\Ol{\overline}

\def\d{\delta}

\def\b{\beta}
\def\dbar{\ov\partial}

\def\a{\alpha}

\def\a{\alpha}

\def\eucl{\hbox{eucl}}
\newtheorem{theorem}{Theorem}[section]
\newtheorem{lemma}[theorem]{Lemma}
\newtheorem{corollary}[theorem]{Corollary}
\newtheorem{proposition}[theorem]{Proposition}

\newtheorem{remark}[theorem]{Remark}
\date{\ }

\begin{document}
\title{\bf  On Point Separation of Bergman Space  on Stein Manifolds with Constant Holomorphic Sectional Curvature 
}

\author{{Xiaojun  Huang}\footnote{Supported in part by  DMS-2247151}
\qquad Song-Ying Li} \maketitle 
\bigskip

\begin{abstract}\vskip 3mm\footnotesize
\noindent   
We prove that the Bergman space of a Stein manifold separates points whenever its Bergman metric is well defined and has non-positive constant holomorphic sectional curvature.
 We construct examples of Stein manifolds  whose Bergman metric is well defined and has positive constant holomorphic sectional curvature, while their Bergman spaces do not separate points. 
We also construct examples of Stein manifolds whose Bergman metric is well defined and has constant scalar curvature, which can  be negative, zero, or positive, yet whose Bergman spaces  do not separate points.

Combined with previously established results in \cite{HuLi1}, this  shows that a Stein manifold cannot admit a well-defined flat Bergman metric, and that it admits a well-defined Bergman metric with negative constant holomorphic sectional curvature if and only if it is biholomorphic to the unit ball of the same dimension, possibly with a pluripolar set removed. 

Our proof is based on H\"ormander’s  estimates for $\overline\partial$- equations; the curvature condition, together with Calabi’s rigidity and extension theorems, is used to construct the required bounded strictly plurisubharmonic functions. The construction of Stein manifolds  with positive constant holomorphic sectional curvature for their Bergman metric is based on classical hyperelliptic Riemann surface theory and its higher-dimensional generalizations.

\vskip 4.5mm

\end{abstract}


\section{Introduction}

Bergman metrics and Bergman kernels of complex manifolds are important invariants in complex analysis and geometry, with profound connections to many central problems in the field. The study of Bergman metrics dates back to the pioneering work of Stefan Bergman \cite{Ber1}, Bochner \cite{Bo}, S.~Kobayashi \cite{Ko}, and other mathematicians. Among the classical results related to the present work is a theorem of Lu \cite{Lu}, which states that a bounded domain  with complete Bergman metric has constant holomorphic sectional curvature if and only if  it is biholomorphic to the unit ball. A remarkable feature of Lu's theorem is that there is no a prior  topological assumption on $M$. Notice that a domain with complete Bergman metric is pseudoconvex or Stein by a classical result of Bremermann \cite{Bre}. However Bergman metrics of many pseudoconvex domains are not complete.

For a complex manifold  of complex dimension $n$, consider the Bergman space $A^2(M)$, consisting of holomorphic $n$-forms  $\phi$  such that  $|\int_M\phi\wedge \ov{\phi}|<\infty$.     The space  $A^2(M)$ is said to
separate points of $M$ if for any  $p_1, p_2\in M$ there exits  $\phi\in A^2(M)$ such that  $\phi(p_1)=0$ and $\phi(p_2)\not =0$.   
 In this paper, the question of when the Bergman space $A^2(M)$   separates points of $M$  is studied, motivated by a classical folklore conjecture that attempts to generalize the Qi-Keng–Lu theorem to pseudoconvex domains with incomplete Bergman metrics. This folklore conjecture states that a bounded pseudoconvex domain has negative constant holomorphic sectional curvature if and only if it is biholomorphic to the unit ball with possibly a pluri-polar set removed. Notable partial results  had been obtained earlier by Dong–Wong \cite{DW1,DW2}. 
In \cite{HuLi1} it was completely  answered affirmatively. In fact, it was proved in \cite{HuLi1}  that a Stein manifold with a well-defined Bergman metric of negative constant holomorphic sectional curvature is biholomorphic to the unit ball with possibly a pluripolar set removed, under the additional assumption that its Bergman space separates points, which holds automatically when $M\subset \CC^n$ is a bounded domain. It was also shown in \cite{HuLi1} that a Stein manifold with a well-defined Bergman metric cannot have a flat Bergman metric if its Bergman space separates points. (This result  was  a joint work of the authors  with J. Treuer). An interesting  question that  was left open from the work of \cite{HuLi1}  is whether the point-separation assumption is actually  needed in these results. The main result of the present paper answers these questions affirmatively. More precisely, we will prove the following:(For further explanation of the terminology used here, see the next section.)

\begin{theorem}\label{mainthm1}
Let $M$ be a Stein manifold.  If  its Bergman metric is well-defined and has a   non-positive    constant holomorphic sectional curvature, then $A^2(M) $ separates points of $M$.
\end{theorem}
 It is known that the holomorphic sectional curvature $c_{\BB^n}$ of the Bergman metric on the unit ball $\BB^n\subset \CC^n$ is $-2/(n+1)$.
Combining Theorem  \ref{mainthm1} with the just discussed  results in \cite{HuLi1}, one obtains the following theorem.

\begin{theorem}\label{mainthm2}
Let $M$ be a Stein manifold of complex dimension $n$.  If  the  Bergman metric of  $M$ is well-defined and has a   non-positive    constant  holomorphic sectional curvature, denoted by $c$, then  $c=c_{\BB^n}=-2/(n+1)$ and $M$ is biholomorphic to $\BB^n\sm E \subset\CC^n$ where $\BB^n$ is the unit ball in $\CC^n$ and $E$ is a pluripolar subset of in $\BB^n$. 
\end{theorem}

We also mention a more recent paper of Ebenfelt--Treuer--Xiao \cite{ETX1, ETX2}, in which the authors proved that the Bergman metric of a bounded (not necessarily pseudoconvex) domain $D\subset\CC^n$ has constant negative holomorphic sectional curvature if and only if $D$ is biholomorphic to a domain of the form $\BB^n\setminus E$, where $E$ is a closed subset of measure zero that is removable for $L^2$-integrable holomorphic functions on $\BB^n\setminus E$.

In the last section, we construct Stein manifolds  whose Bergman metric has positive constant holomorphic sectional  curvature, while its Bergman space fails to separate points. We also construct Stein manifolds of higher dimension whose Bergman metric has constant scalar curvature, which can  be positive, zero, or negative, although their Bergman spaces do not separate points. More precisely, we prove the following:
\begin{proposition}\label{333}
For every \(n \in \mathbb{N}\), there exists a Stein manifold \(M\) of complex dimension \(n\), embedded in ${\mathbb C}^{n+1}$,  such that its Bergman space does not separate points in \(M\), yet the holomorphic sectional curvature of its Bergman metric is a positive constant. Moreover, there exist  Stein manifolds whose Bergman metric is well defined but whose Bergman space does not separate points, and whose scalar curvature can  be a positive, negative, or zero constant.
\end{proposition}

In \cite{W}, Wiegerinck first constructed domains in $\mathbb{C}^2$
with nontrivial finite-dimensional Bergman spaces. Motivated by a
different line of investigation, we proved in \cite{HuLi1} that a
complex manifold whose Bergman metric is well defined and has constant
positive holomorphic sectional curvature must have a finite-dimensional
Bergman space, provided that its Bergman space separates points. In
light of this result, Wiegerinck's construction became particularly
relevant to our search for domains whose Bergman metrics have constant
positive holomorphic sectional curvature and provided an important
source of inspiration for our work in \cite{HuLi1}. However, the
examples in \cite{W} have Bergman forms that are degenerate everywhere,
whereas in \cite{HuLi1} we introduced a substantially different
construction that yielded the first example of a domain with a
positive-definite Bergman metric of constant positive holomorphic
sectional curvature.

Motivated by the work of \cite{W} and \cite{HuLi1},  additional nice examples of this type  were subsequently constructed in \cite{HK}, \cite{Ya}, and \cite{BGNX}. In particular, the authors of \cite{BGNX} constructed examples of non-pseudoconvex domains carrying Bergman metrics with every permissible positive constant holomorphic sectional curvature.

The examples constructed in this paper cannot be realized as domains in complex Euclidean space, although they are Stein manifolds. Indeed,  we conjectured in \cite{HuLi1} that no pseudoconvex domain in complex Euclidean space admits a well-defined Bergman metric with positive constant holomorphic sectional curvature. The conjecture is supported by a theorem established in \cite{HuLi1}, which asserts that the Bergman space of any complex manifold carrying such a Bergman metric is finite-dimensional, together with the  Wiegerinck conjecture \cite{W}, according to which the Bergman space of a pseudoconvex domain is either zero-dimensional or infinite-dimensional.

Bergman metrics have been extensively studied for almost a century since Bergman’s foundational work, yet the theory remains far from complete, and any new result about these metrics provides important motivation for further study in complex analysis and geometry. Here we mention the work in \cite{Fe}, \cite{Mo1}, \cite{Mo2}, \cite{FW},
as well as many references therein on Bergman metrics and Bergman kernels.  We also mention recent related studies such as in the work of  \cite{GZ}, \cite{L},  \cite{KLS}, \cite{MZ} and references therein. This is especially true beyond the bounded domain case, where much current work in complex analysis and geometry aims to extend to unbounded domains and more general Stein  manifolds.

\bigskip
\noindent\textbf{Acknowledgment.}
\textbf{Acknowledgement.}
The authors are  grateful to N. Mok for many stimulating discussions over the years on Calabi's theory, Bergman geometry, and their  interaction  with other areas of. 

\section{Preliminaries}

Let $M$ be a complex manifold of complex dimension $n\ge 1$.
Write $\O^{n}(M)$  for the space of  holomorphic $n$-forms on $M$ and define the Bergman space of $M$ to be
\begin{equation}
A^2(M):=\Big\{f\in
\O^n(M):\ \  b_n \int_{\Omega} f\wedge
\overline f<\infty \Big\},\ \  \h{where}\ \  b_n=( i)^n (-1)^{\frac{n^2-n}{2}}=i^{n^2}.
\end{equation}
Notice  that the constant $b_n$ is chosen so that
 for a $L^2$ integrable holomorphic function $g$ in $\CC^n$ and $f=gdz_1\wedge\cdots \wedge dz_n$, $ b_n\int_{\CC^n} f\wedge
\ov f= \int_{\CC^n}|g|^2 d{\rm{vol}}$. Here we write $z_j=x_j+iy_j$ and $d{\rm{vol}}=2^n dx_1\wedge dy_1\wedge \cdots \wedge dx_n\wedge dy_n$.
  
  $A^2(M)$ is a Hilbert space with
  inner product defined by:
\begin{equation}
(f, g)=b_n\int_{M}f\wedge\overline g,\quad
 ~\text{for all}~f, g\in A^2(M).
 \end{equation}
  Assume  $A^2(M)\not = 0$. Let $\{f_j\}_1^N$ be an orthonormal
 basis of $A^2(M)$  with $N\le \infty$ and define the Bergman kernel, which is an $(n,n)$-form (see, e.g., \cite{Ko}), to
 be $K_M=b_n\sum_{j=1}^N f_j\wedge\overline f_j$. In a local holomorphic coordinate
  chart $(U, z)$ on $\Omega$, we have
\begin{equation}
K_M=b_n k_M(z, \overline z) dz_1\wedge\cdots\wedge dz_n\wedge
d\overline z_1\wedge\cdots\wedge d\overline z_n~\text{ in }~ U.
\end{equation}
Assume that  $A^2(M)$ is base-point free in the sense that $K_M$ is nowhere zero on $M$.
We then well define a semi-positive  Hermitian $(1, 1)$-form on $M$ by
$\omega^B_M=\sqrt{-1}\partial\overline\partial \log k_M(z, \overline z).$
We call $\omega^B_M$ the Bergman form of $M$.  If it induces a positive definite metric on $M$, then it is a K\"ahler metric and is  called the Bergman metric of $M$:
$$d^2s:=\sum_{j,k=1}^n\frac{ \partial ^2 }{\partial z_k\overline\partial z_j}\left( \log k_M(z, \overline z)\right) dz_k\otimes d\overline{z_j}.$$

 Bergman space $A^2(M)$ is said to {\it separate points} if for
$p_1,p_2\in M$ with $p_1\not =p_2$, there is a
holomorphic $n$-form $\phi\in A^2(M)$ such that $\phi(p_1)=0$ but $\phi(p_2)\not
=0$.

We say that \(A^2(\mathcal{O})\) separates holomorphic directions at \(p \in M\) if, for every nonzero \(X_p \in T_p^{(1,0)}M\), there exists \(\phi \in A^2(M)\) such that \(\phi(p)=0\) and \(X_p(\phi)(p)\neq 0\), where \(\phi\) is identified with a holomorphic function near \(p\) in a local holomorphic coordinate chart. Kobayashi \cite{Ko} proved that the Bergman metric on \(M\) is well defined near \(p\), meaning that the Bergman form \(\omega_M^B\) is positive definite in a neighborhood of \(p\), if and only if the Bergman space \(A^2(M)\) is base point free at \(p\), that is, there exists \(\phi \in A^2(M)\) with \(\phi(p)\neq 0\), and separates holomorphic directions at \(p\). The following is immeidate from the definition:

\begin{proposition}\label{99-01}
Let \(M\) be a complex manifold. Suppose its Bergman metric exists at a certain point \(p \in M\), that is, the Bergman form \(\omega_M^B\) is positive definite in a neighborhood of \(p\). Then there exists a  complex analytic hypersurface \(E \subset M\) such that the Bergman metric of \(M\) exists on \(M \setminus E\), that is, the Bergman form \(\omega_M^B\) is positive definite on \(M \setminus E\). Here $E$ could be an empty subset.
\end{proposition}

\begin{remark}\label{rem:2.2}
A subset $E \subset M$ is called a \emph{complex analytic hypersurface} of $M$ if $E$ is a closed (possibly singular) complex analytic subvariety of $M$ of pure codimension one at every point.
We also note that, by the results of \cite{GHH17,HJL25}, if a pseudoconvex domain has a boundary point that is a local peak point for holomorphic functions continuous up to the boundary in a neighborhood of that point, then its Bergman metric is well defined in a neighborhood close to  this boundary point. Consequently, by Proposition~\ref{99-01}, after possibly removing a complex analytic hypersurface, the Bergman metric is well defined on the remaining domain. In particular, this applies whenever the pseudoconvex domain has a $C^2$-smooth strongly pseudoconvex boundary point.

By a classical result of
Docquier and  Grauert \cite{DG}, when $M$ is Stein, $M\sm E$  is also Stein.
\end{remark}

\begin{proof}[Proof of Proposition~\ref{99-01}]
Since the Bergman metric is well defined at \(P\), the Bergman space \(A^2(M)\) is base-point free and separates holomorphic tangent directions at \(P\).

Choose a local holomorphic coordinate chart \((U,z)\) centered at \(P\) such that \(z(P)=0\). We identify any holomorphic \(n\)-form
 $
\varphi\,dz_1\wedge\cdots\wedge dz_n$
with its coefficient function \(\varphi\).

First, choose
$\varphi_0\in A^2(M)$ 
such that
$\varphi_0(P)\neq 0.$ 
Now let
\[
\{\varphi_0,\varphi_1,\varphi_2,\ldots\}
\]
be an orthonormal basis of \(A^2(M)\) with
\[
\varphi_j(P)=0,\qquad j\ge 1.
\]

For each \(j=1,\ldots,n\), define
\[
\beta_j
=
\left(
\frac{\partial \varphi_1}{\partial z_j}(0),
\frac{\partial \varphi_2}{\partial z_j}(0),
\ldots
\right)
\in \ell^2.
\]

We claim that
$$
\{\beta_1,\ldots,\beta_n\}
$$
are linearly independent.

Indeed, suppose that there exist constants
$\lambda_1,\ldots,\lambda_n,$
not all zero, such that
$\sum_{j=1}^n \lambda_j \beta_j=0.$
Then
$$
\sum_{j=1}^n \lambda_j \frac{\partial \varphi_\ell}{\partial z_j}(0)=0,
\qquad \forall\, \ell\ge 1.
$$ 
Hence every element of \(A^2(M)\) vanishing at \(P\) has vanishing derivative in the  non-zero holomorphic direction
$X=\sum_{j=1}^n \lambda_j \frac{\partial}{\partial z_j}|_0$
at \(P\), contradicting the assumption that \(A^2(M)\) separates holomorphic directions. Therefore,
$\beta_1,\ldots,\beta_n$
are linearly independent.

For each positive integer \(N\), write
$$ \beta_j=(\beta_{j1},\cdots, \beta_{jm},\cdots)\ \hbox{and}\ \ 
\beta_j^{(N)}=(\beta_{j1},\ldots,\beta_{jN}),
$$
and define
$$
V_N=
\left\{ c=(c_1,\ldots,c_n)\in \mathbb{C} ^n:
\sum_{j=1}^n c_j \beta_j^{(N)}=0 \right\}.
$$ 
Then
$
V_1\supset V_2\supset \cdots.$ 
Since each \(V_N\) is a subspace of the finite-dimensional space \(\mathbb C^n\), there exists \(N_0\) such that
$ 
V_N=V_{N_0},\qquad N\ge N_0.
$ 

If
$ 
c=(c_1,\ldots,c_n)\in V_{N_0},$ 
then \(c\in V_N\) for every \(N\ge N_0\). Consequently,
$\sum_{j=1}^n c_j \beta_j=0.$ 
Since the vectors \(\beta_1,\ldots,\beta_n\) are linearly independent, it follows that
$  
V_{N_0}=\{0\}.
$
Hence, after possibly renumbering the basis, we may assume that
$ 
\det
\left(
\frac{\partial \varphi_j}{\partial z_k}(0)
\right)_{1\le j,k\le n}
\neq 0.
$ 

Now define
$$
E_0=\{\varphi_0=0\}.$$
For each \(q\in M\), choose a holomorphic coordinate system \((U_q,z)\) with \(z=(z_1,\ldots,z_n)\) on a neighborhood \(U_q\) of \(q\), and set
\[
E_1\cap U_q
=
\left\{
z\in U_q:
\det
\left(
\frac{\partial \varphi_j}{\partial z_k}(z)
\right)_{1\le j,k\le n}
=0
\right\}.
\]
Then \(E_1\) is a well-defined  complex analytic hypersurface of \(M\), which might be empty.

Therefore
$ 
E=E_0\cup E_1
$ 
is either empty or a complex analytic hypersurface  of \(M\). Moreover, the Bergman form is positive definite on
$ 
M\setminus E,$ 
or equivalently, the Bergman metric is well defined on
$M\setminus E.$ 

\end{proof}

For the Bergman metric to be well-defined, it is not necessary for $A^2(M)$ to separate points; however, this property holds automatically in many cases, such as when $M$ is a bounded domain in $\mathbb{C}^n$. Indeed, in this paper, we investigate how curvature conditions influence the point separation property of the Bergman metric. This study is motivated by the authors’ earlier work \cite{HuLi1} on complex manifolds with constant holomorphic sectional curvature for their Bergman metric, in which point separation was assumed.





\bigskip
We define the infinite dimensional projective space, denoted by $\PP^\infty$,  as
$$\PP^\infty:=\ell^2\sm \{0\}/\sim, \ \h{where }\ell^2:=\{(x_1,\cdots,x_m,\cdots): \sum_{j=1}^\infty |x_j|^2<\infty\}.$$  Here, for any given  $x,y\in \ell^2\sm \{0\}$,
$x\sim y $ if and only if $x=ky$ for a
certain $k\in \CC\sm\{0\}.$ For $x=(x_1,\cdots,x_n,\cdots)\in
\ell^2\sm \{0\}$, we write $[x_1,\cdots,x_n,\cdots]$ for its
equivalence class in $\PP^\infty$, called the homogeneous coordinate
of the equivalence class of $x$. $\PP^m$ is naturally identified as
a closed subspace of $\PP^\infty$ by adding zeros to homogeneous
coordinates of  $\PP^m$.

A map ${\cal F}: M\ra \PP^\infty$  is called a holomorphic map from
$M$ into $\PP^\infty$ if there is a holomorphic representation $[F]$
of ${\cal F}$  near each $p\in M$. Namely, for each $p\in M$, there
is a small neighborhood $U_p$ of $p$ in $M$ such that ${\cal F}=[F]$
with $F=(f_1,\cdots, f_m, \cdots)$, where each $f_j$ with $j\in \NN$
is a holomorphic function in $U_p$ and for each $q\in U_p$, there is
a $k$, which may depend on $q$, such that $f_{k}(q)\not = 0$.
Moreover, $\sum_{j=1}^\infty |f_j(z)|^2$ is required  to  converge
uniformly on compact subsets of $U_p$.
We   denote by $$X_{\F}=[X_F]\subset
\PP^n$$  with $X_F$ being the closed linear subspace of $\ell^2$  generated by the linear span of  $F(M)$ in $\ell^2$. 
 $X_{\F}$ is independent of the choice of local holomorphic
representations of $\F$.

The formal Fubini–Study metric on $\PP^\infty$, denoted by  $\o_{st}$, with homogeneous coordinates $$[z_1, \dots, z_n, \dots],$$ is formally defined by

\begin{equation}\label{(1.1)}
\omega_{st}=i\p \dbar \log  \Big(\sum_{j=1}^\infty |z_j|^2\Big).
\end{equation}
Let \begin{equation}\label{(1.2)} \F: \ (M, \omega) \to (
\PP^\infty, \omega_{st})
\end{equation}
be a holomorphic map, where $\o$ is a  K\"ahler metric over $M$. If
for any  local holomorphic representation
$\F=[f_1,\cdots,f_n,\cdots]$ over $U\subset M$, it holds that
\begin{equation}\label{(1.3)}
\o=\F^*(\omega_{st}):=i\p \dbar \log \Big(\sum_{j=1}^\infty
|f_j(z)|^2\Big)\quad  \h{over } U,
\end{equation}
we call  $\F$  a local holomorphic isometric embedding from $M$ into
$\PP^\infty$.



\section{Negative constant holomorphic sectional curvature}
Let $M$ be a complex manifold and assume that $A^2(M)$ is base-point free.  For an
orthonormal basis $\{\phi_j\}_{j=1}^{N}$, we then have a well-defined
holomorphic map $$\B:=[\phi_1,\cdots,\phi_N,0,\cdots],\  N<\infty;\ \ \h{or } \B:=[\phi_1,\cdots,\phi_m,\cdots]\ \h{for }N=\infty$$ from $M$ into
$\PP^\infty$. We also write $B=(\phi_1,\phi_2,\cdots)$. Thus $\B=[B]$. $\B$ is called a Bergman-Bochner or a cannonical  map and is a local holomorphic isometric embedding from $M$ into $\PP^\infty$.
 Then $A^2(M)$ separates points if and only if any Bergman-Bochner
 map  ${\cal B}: M \to \PP^\infty$ is one-to-one. 
 Suppose that the holomorphic sectional curvature of $\o_M$ is a
negative constant on $M$. Then for  $p_0\in M$,  by  Theorem 6 in [Bo], there is a biholomorphic map $F$
from a small neighborhood  $U$ of $p_0\in M$ to a subdomain in the unit ball $\BB^n\subset \CC^n$ such that
$F^*(\ld\o_{\BB^n})=\o_M$ for a certain positive constant $\ld$. 

Now one can construct as in \cite{HuLi1} 
 a one-to-one  holomorphic isometric embedding $\mathcal{T}=[1,
P_1, \cdots, P_j, \cdots]$ from $(\BB^n, \mu\omega_{\BB^n})$ into
$\mathbb{P}^{\infty},$ where $P_j's$ are holomorphic monomials in $z\in \CC^n$. Moreover  $P_j=a_jz_j$ with  $a_j\in \CC \not =0$ for $j=1,\cdots,n$. We similarly write   $T=(1,
P_1, \cdots, P_j, \cdots).$   
Then ${\mathcal T}\circ F: U\to \PP^\infty$ is a one to one
holomorphic isometric embedding.
By the Calabi rigidity theorem
(\cite{Ca} or \cite{HuLi1}), there is a one to one isometric
isomorphism ${\cal L}: X_\B\ra X_{\mathcal T}$ such that
\begin{equation}\label{X1}
    {\cal L}\circ {\mathcal B}(z)={\mathcal T}\circ F(z),\quad z\in U.
\end{equation}
 Notice that ${\cal L}=[L]$ is induced by a linear isometric isomorphism 
 ${ T}$ from the closed linear subspace $X_{B}\subset \ell^2$ to  $X_{ T}\subset\ell^2$. By the
Calabi extension theorem (\cite{Ca} or \cite{HuLi2}), ${\mathcal
T}\circ F$ extends along any curve in $M$ initiated from $p_0\in U$ to a local holomorphic isometric embedding. Since each $z_j$ is a
component in ${\mathcal T}$ up to a coefficient, we see that $F$
extends holomorphically and isometrically along any curve in $M$
starting at $p_0$ to $\wt{F}$.  $\wt{F}$ a priori might be
multi-valued. 
Since the K\"ahler  metric $\ld\o_{\BB^n}$ is
complete, $\wt{F}$ has image in $\BB^n$ and thus (\ref{X1}) shows that $\wt{F}$, still denoted by $F$ in the following, is globally defined in $M$, which is a local biholomorphism. 

  
Note that  (\ref{X1}) holds over $M$. Hence for any $p_1,p_2\in M$, $\mathcal B(p_1)=\mathcal B(p_2)$ if and only if $F(p_1)=F(p_2)$. 

\begin{proposition}\label{1-3} 
Let  $M$ be a Stein manifold of complex dimension $n\ge 1$.  
  Assume further that there is a smooth strictly plurisubharmonic function  $\sigma$ on $M$ 
  that is bounded from above.  Then $A^2(M)$ separates points of $M$.
\end{proposition}
The above proposition together with the result in \cite{HuLi1} gives the following:
\begin{theorem}\label{mainthm3}
Let $M$ be a Stein manifold of complex dimension $n$.  If  the  Bergman metric of  $M$ is well-defined and has a   negative   constant  holomorphic sectional curvature, denoted by $c$,  then $M$ is biholomorphic to $\BB^n\sm E \subset\CC^n$ where $\BB^n$ is the unit ball and $E$ is a pluripolar set of $\CC^n$. Moreover, 
 $c=c_{\BB^n}=-2/(n+1)$.
\end{theorem}
Our proof is based on Hörmander’s \(L^2\) method for solving \(\bar\partial\)-equations on Stein manifolds. Since we could not find a suitable reference leading to  the proof of Proposition~\ref{1-3}, we include a detailed argument here.

We first recall the classical Bochner--Kodaira--Morrey--Kohn identity on a smoothly compactly supported domain in a K\"ahler manifold.

\medskip

Assume that  $(M,\omega)$ is  a K\"ahler  manifold. Let $\O\subset \subset M$ be a  smoothly bounded relatively compact domain with   $\rho\in C^\infty (\ov{\O})$ being a defining function. Namely, $\rho|_{\O}<0$ and $d\rho|_{\p \O}\not =0$. By multiplying a positive smooth function over $\ov \O$ to $\rho$, if needed, we  assume that $|d\rho|\equiv 1$ in a small neighborhood $\p \O$ in $\ov{\O}$.   Let $E$ be a Hermitian holomorphic line bundle over $M$. Write $\ov\p^{E}$ for the $\ov\p$-operator sending an $E$-valued $(0,q)$-form  $\a\in \Lambda^{0,q}(\Omega, E)$ to  an $E$-valued $(0,q+1)$-form:  $\ov\p\a\in \Lambda^{0,q+1}(\Omega, E)$ .  Still write  $\ov\p^{E}$ for its closed extension  from  $L^2_{(0,q)}(\Omega,E)$  into $L^2_{(0,q+1)}(\Omega,E)$, which is then a  densely defined closed operator. Write $\ov\p^{E *}$ for the Hilbert adjoint operator of $\ov\p^{E}$.

Let $\phi\in \L^{(0,1)}(\ov\O,E)\cap \h{Dom}({\ov\p}^{E*})$. Then we have the following special case of the  classical Bochner-Kodaira-Morrey-Kohn formula (see, for instance,  [pp 119, Siu] or  [Theorem 1.4.21, \cite{MaMa}]):

\begin{equation}\label{BKMHK} 
\|\ov\p^{E}\phi\|_{\O,h}^2+\|\ov\p^{E *}\phi\|_{\O,h}^2=\|\ov{\nabla}^E\phi\|_{\O,h}^2+ (\h{Ric} (\phi), \phi)_{\O,h} + (K(\phi), \phi)_{\O,h}+(\h{Levi}_\rho(\phi), \phi)|_{\p\O, h}.
\end{equation}
We explain the notations used above as follows:
Let $(U,(z_1,\cdots,z_n))$ be a holomorphic coordinate patch of $M$, over which $E$ is also trivialized.  Write $e$ for a basis of $E$ over $U$. Write $g=\sum_{\a \b}g_{a\ov\b}dz_\a \otimes d \ov{z_\b}$ for the K\"ahler metric tensor over $U$ with the associated K\"ahler form $\omega=i \sum_{\a \b}g_{a\ov\b}dz_\a \wedge d \ov{z_\b}$. Write   $g=\h{det}(g_{k\ov l})$ and  $e^{-h}=<e,e>$ with respect to the Hermitan product in $E$. Over $U$ the Ricci tensor  $R_{\ov v}^{\ov \a}$ of $M$ and the curvature tensor $K_{\ov v}^{\ov \a}$ of $E$ are given by
$$R_{\l \ov v}=- \p_{\l} {\p}_{\ov v}\log (g),
\ \ \  K_{\l \ov v}=- \p_{\l} {\p}_{\ov v}\log (e^{-h}).
$$
The Levi form of $\rho$, denoted by  $\h{Levi}_\rho:=\p_\a\ov{\p_\b}\rho dz_\a\otimes d\ov{z_\b}$, is the  complex Hessian of $\rho$ but restricted to the complex tangent bundle of type $(1,0)$ along $\p \O$. $\nabla^E$ is the unique Hermitian connection on $\L^{(0,q)}(E)$ for $q\ge 0$ 
associated with  the K\"ahler metric of $M$ and
the Hermitian metric of $E$.

Write $\phi=\phi_{\ov j}d\ov{z_j}  \otimes  e$. 
Then
  \begin{equation}\label{BKMK-02}
\begin{split}
 &(\h{Ricc} (\phi), \phi)_{\O,h} =
 \int_{\O} R_{ \a \ov\b}{\phi^{\a}}\ov{\phi^{\b}} e^{-h}\frac{\omega^n}{ n!}, \\ 
 & (K(\phi), \phi)_{\O,h}=\int_{\O}
 K_{ \a \ov\b}{\phi^{\a}}\ov{\phi^{\b}}
 e^{-h}\frac{\omega^n}{n!}, \\
 &(\h{Levi}_\rho(\phi), \phi)|_{\p \O,h}=\int_{\p \O} g^{s\ov t}\p_s\p_{\ov{j}}(\rho){\phi_{\ov t}}\ov{\phi^{ j}} e^{-h} \h{d} \h{vol}(\p \O)
\\ &
\|\overline{\nabla}^E\phi\|_{\Omega,h}^{2}
= \int_{\Omega}
|\overline{\nabla}^E\phi|^{2}\ \frac{\omega^n}{ n!}
= \int_{\Omega} g^{k\overline{\ell}}\langle{\nabla}_{\overline \ell}^E\phi, {\nabla}_{\overline k}^E\phi\rangle \frac{\omega^n}{ n!}.
 \end{split}
 \end{equation}
Here  $\phi^t=g^{\ov{\l}t}\phi_{\ov \l}$,  $\h{d} \h{vol}(\p \O)$ is the induced volume form  on $\p \O$. Namely,  $\h{d} \h{vol}(\p \O)=i_{L}(\frac{\omega^n}{n!}) $, the contraction of $\frac{\omega^n}{ n!}$ with the unit outward normal vector of $\p \O$.

\begin{proof} [Proof of Proposition \ref{1-3}]
Let $\rho$ be a strictly smooth p.s.h exhausting  function of $M$ and choose a increasing sequence $\{r_j\}_{j=1}^{\infty}\subset \RR$ such that each $r_j$ is a regular value of $\rho$ and $p_1,p_2\in \O_{r_j}: =\{p\in M:\ \rho(z)<r_j\}$ for each $j\ge 1$.
Since $M$ is Stein, it can be properly embedded into ${\mathbb C}^{2n+1}$. Pull back the restricted Euclidean metric by this embedding, we get  a complete real analytic  K\"ahler  metric $d^s$ with $\o$ on $M$ as the associated positive $(1,1)$ metric tensor.

Let
$ 
d^2s=g_{i\bar j}\,dz^i\otimes d\bar z^j$
be a Kähler metric in a local holomorphic coordinate chart $(U,z)$.
We have the naturally  induced Hermitian metric $h_{K_M}$ on the canonical line bundle
$K_M=\Lambda^{n,0}T^*M$, which is given by  
\[
\left<dz^1\wedge\cdots\wedge dz^n,
\,dz^1\wedge\cdots\wedge dz^n\right>_{h_{K_M}}
=\frac{1}{\det(g_{i\bar j})}.
\]
Equivalently,
\[
|a\,dz^1\wedge\cdots\wedge dz^n|_{h_{K_M}}^2
=\frac{|a|^2}{\det(g_{i\bar j})}.
\]

 Let $\kappa=\log\det(g_{i\bar j})$  be the smooth function on $U$ such that 
 \[<dz_1\wedge \cdots \wedge dz_n,  dz_1\wedge \cdots \wedge dz_n>_{h_{K_M}}:=<dz_1\wedge \cdots \wedge dz_n,  dz_1\wedge \cdots \wedge dz_n>_\kappa=\frac{1}{\det(g_{i\bar j})}=e^{-\kappa}.\]

We set $E=K_M$ to be the canonical line  bundle equipped with the just mentioned  Hermitan metric.

Notice that for a holomorphic section $s$ of $E_j=E|_{\O_{r_j}}$  writing in the  holomorphic coordinate chart  mentioned above: $s=adz_1\wedge \cdots \wedge dz_n$, then we see that
$$(s,s)_{\kappa}:=\int_{\O_{r_j}}<s,s>\frac{\o^n}{n!}=i^{n^2}\int_{{\O_{r_j}}}s\wedge\ov{s}.$$

Let $p_1\not =p_2\in M$.  There are   certain neighborhoods $U(p_j)$ of $p_j$  for $j=1,2$  such that $\Phi_j$ maps biholomorphically  $\ov{U(p_j)}$  to the  closed Euclidean ball $\ov{B(0,1)}$ centered at the origin  with radius $1$ for $j=1,2$, respectively. Moreover, $\ov{U(p_1)}\cap \ov{U(p_2)}=\emptyset$. Write its inverse to be $G_{p_j}:  B(0,1)\ra U(p_j)$ for $j=1,2$, respectively.
 Let $\chi_0$ be a non-negative function on $M$ with support in the unit of two coordinate balls $G_{p_1}(B(0, \frac{1}{2}))\cup G_{p_2}(B(0, \frac{1}{2}))$
 and 
$$\chi_0(p)\equiv 1,  \quad\hbox{ over  } 
G_{p_1}(B(0, \frac{1}{4}))\cup G_{p_2}(B(0, \frac{1}{4})).
$$
For $0<\e<<1$, define
$$h^*_{\e}(p)= n \chi_0(p)\log \Big(  |\Phi_1(p)-\Phi_1(p_1)|^2+\e\Big)+n \chi_0(p)\log \Big(  |\Phi_2(p)-\Phi_2(p_2)|^2+\e\Big)+m \sigma(p).
$$
Then when $m$ is sufficiently large,  that  depends only   on $\chi_0$,
 $h_{\e}^*(z)$ is  smoothly strictly plurisubharmonic in $M$ uniformly  bounded from above with repsect to $\epsilon$.
 
 We now identify  $p\in  U(p_j)$ with its holomorphic coordinate $z=\Phi_j(p)$ for $j=1,2$. Then by making $m$ sufficiently large, we  assume that there is a constant $C_1>0$, independent of $\e$ but only of  $\chi_0$,  such that for $\xi\in\CC^n$,
$$\sum_{i,j=1}^n\frac{\p^2 h^*_\e}{\p z_i\p \ov{z_j}} \xi_i\ov{\xi_j}\ge C_1 |\xi|^2 \  \h{over } U(p_1)\cup U(p_2).$$

 

We let $A$ be the multiplication operator by the characteristic function 
$$\chi_{\{G_{p_1}(B(0, \frac{1}{2}))\cup G_{p_2}(B(0, \frac{1}{2}))\}},$$  namely,  for each $\phi \in L^2_{(0,1)}(\O_{r_j}, E)$, $$A(\phi)=\chi_{\{G_{p_1}(B(0, \frac{1}{2}))\cup G_{p_2}(B(0, \frac{1}{2}))\}} \phi$$ which is $\phi$ in $G_{p_1}(B(0, \frac{1}{2}))\cup G_{p_2}(B(0, \frac{1}{2}))$ and $0$ otherwise.
  Notice that  $\p \O_{r_j}$ is  strongly pseudoconvex and thus the Levi form is positive definite. Now we equip  $K_M$ with the metric  $$e^{-(h_\epsilon ^*+\kappa)}:=e^{-h_\epsilon},\ \ i.e.,\ 
|a\,dz^1\wedge\cdots\wedge dz^n|_{h_\epsilon}^2
=|a|^2 e^{-(h_\epsilon ^*+\kappa)}.  $$

Notice that in a local chart $(U,z=(z_1,\ldots,z_n))$,

\[
R_{\lambda\bar{\nu}}
=
-\partial_\lambda\partial_{\bar{\nu}}
\log\det(g_{j\bar{k}}),\ \ 
K_{\lambda\bar{\nu}}
=
-\partial_\lambda\partial_{\bar{\nu}}
\log\!\left(e^{-h_\epsilon}\right).
\]
Thus
\[
R_{\lambda\bar{\nu}}+K_{\lambda\bar{\nu}}
=
-\partial_\lambda\partial_{\bar{\nu}}
\log\det(g_{j\bar{k}})
+
\partial_\lambda\partial_{\bar{\nu}}
\bigl\{\log  \det(g_{j\bar{k}})+h^*_\epsilon\bigr\}
=
\partial_\lambda\partial_{\bar{\nu}}
h_\epsilon^{*}.
\]

The Bochner-Kodaira-Morrey-Kohn formula gives  the following estimate

\begin{equation}\label{BKMHK-02} 
\begin{split}
&\|\ov\p^{E} \phi\|_{\O_{r_j},h_\e}^2+\|\ov\p^{E *}\phi\|_{\O_{r_j},h_\e}^2\\ &=\|\ov{\nabla}^E\phi\|_{\O_{r_j},h_\e}^2+ (\h{Ricc} (\phi), \phi)_{\O_{r_j},h_\e} + (K(\phi), \phi)_{\O_{r_j},h_\e}+(\h{Levi}_\rho(\phi), \phi)|_{\p \O_{r_j},h_\e}\\
&\ge  \int_{\O_{r_j}} 
g^{\bar{\alpha}\lambda}
\bigl(R_{\lambda\bar{\nu}}+K_{\lambda\bar{\nu}}\bigr)
\varphi_{\bar{\alpha}}
\overline{\varphi^{\nu}}e^{-h_\epsilon}\\
&\ge   \int_{U(p_1)\cup U(p_2)} \sum_{\lambda,\nu=1}^{n}\p_\lambda\p_{\ov \nu}h_{\e}^*(p) \phi^{ \lambda}\overline { \phi^{ \nu}}\ e^{-h_{\e}(p)} d\l \\
&\ge C_1\|A(\phi)\|_{h_\e}^2.
\end{split}
\end{equation}
Let $\beta^*$ be a holomorphic $n$-form over $M$ with $\beta(p_j)\not =0$, $j=1,2$. Shrinking $U(p_j)$ if needed, we assume that $\beta^*\not =0$ over $U(p_1)\cup U(p_2).$

Define $\beta  = \chi_0(p) \beta^*$ over $U(p_1)$ and to be zero otherwise. Then $\beta$ is a smooth section of $E$  compactly supported  in   $G_{p_1}(B(0, \frac{1}{2}))$.
Then $\eta:=\ov\p{\beta}\in \L_c^{(0,1)}(\O_{r_j}, E)$ with $\ov\p \eta=0$.
Notice that  $A\eta=A^*\eta= \eta$.

By the H\"ormander smoothing lemma \cite{Hor},  (\ref{BKMHK-02}) holds for any $\phi\in \h{Dom}(\ov\p^{E}|_{\O_{r_j}})\cap  \h{Dom}(\ov{\p}^{E *}|_{\O_{r_j}})$.  By Theorem 1.1.4 of \cite{Hor} and the regularity of solutions of the  $\ov\p$-equation,
 we can find a $\psi_{\epsilon, j}\in L^2(\O_{r_j},E)\cap \L^{(0,0)}(\O_{r_j}, E)$ such that
$\ov \p \psi_{\epsilon, j}=\eta$ with $$\|\psi_{\epsilon, j}\|^2_{\O_{r_j},h_\e}\le \frac{1}{C_1}\|\ov \p \eta\|^2_{\O_{r_j},h_\e}=\frac{1}{C_1}\int_{U(p_1, 1)\sm U(p_1,\frac{1}{4})}<\ov\p \eta, \ov\p \eta> \frac{\omega^n}{n!}.$$
Here $U(p_1,\delta)=\Phi_1(B(0,\delta) $  for $0<\delta<1$.
Notice that   it holds that
\begin{eqnarray}
e^{-h^*_\e}&=& e^{-m\sigma} {1\over \big((|\Phi_1(p)-\Phi_1(p_1)|^2+\e)(|\Phi_2(p)-\Phi_2(p_2)|^2+\e)\big)^{n\chi_0(p)}}\ge C_2 \hbox{ over }M,
\end{eqnarray}
\begin{eqnarray}
e^{-h^*_\e} \le  C_3 \quad \hbox{over}\  U(p_1)\sm U(p_1,\frac{1}{4}),
\end{eqnarray} 
and
\begin{eqnarray}
\|\psi_{\epsilon, j}\|_{\O_{r_j},h_\e}^2&=&\int_{\O_{r_j}}
 <\psi_{\epsilon, j},  \psi_{\epsilon, j}>_{h_\epsilon}\frac{\o^n}{n!}= \big| \int_{\O_{r_j}}
 \psi_{\epsilon, j}\wedge \overline{\psi_{\epsilon, j}}e^{-h^*_\e}\big| \nonumber\\
 &\ge & \big| C_2\int_{\O_{r_j}}
 \psi_{\epsilon, j}\wedge \overline{\psi_{\epsilon, j}}\big|, \ \hbox{as}\  <\psi_{\epsilon, j},  \psi_{\epsilon, j}>_{h_\epsilon}\frac{\o^n}{n!}= b_n\psi_{\epsilon, j}\wedge \overline{\psi_{\epsilon, j}}e^{-h^*_\e}.
\end{eqnarray}
Here  $C_2, C_3$ are  certain positive numbers   independent of $j$ and $\e$. Hence, we get that  the sequence$\{\psi_{\e,j}\}_{j=1}^{\infty}$ has a uniformly bounded  Bergman norm ($L^2$ norm).  Now, by a standard use of the normal family argument to $\{\psi_{\e,j}-\psi_{\e,\ell}\}_{j\ge \ell}$ and  by passing to a subsequence if needed, we can assume
$\{\psi_{\e,j}\}_j$
converges uniformly on any compact subset of $M$ to $\psi_\e$  such that $\overline{\partial }\psi_\e=\eta=\overline{\partial}\beta$ and  $\|\psi_\e\|^2_{M,h_\e}\le C$ with $C$ independnet of $\e$. Letting $\e\ra 0^+$ and passing to a subequence, we can assume $\psi_\e- \beta$ converges, along a certain subsequence,  uniformly on compact subsets of $M$ to $\psi-\beta\in A^{2}(M)$
with $\ov\p \psi=\eta$.
Write $\psi=\psi_0 \beta^*$ in $U(p_1)\cup U(p_2)$. ( Note that   $\beta^*\not =0 $ over $\overline{U(p_1)\cup U(p_2)}$. 
Applying the Fatou lemma,
we get
$$\int_{U_(p_1)\cup U(p_2)} \frac{|\psi_0|^2  e^{-\sigma}}{ \big(|F(p)-F(p_1)||F(p)-F(p_2)|\big)^{2n \chi_0}}\frac{\omega^n}{n!}\le C$$
which forces $\psi_0(p_1)=0$ and also $\psi_0(p_2)=0$.  

Thus $\gamma=\psi- \beta \in A^2(\Omega)$. We then have $\gamma(p_1)=-\beta(p_1)\not = 0$ but $\gamma (p_2)=\beta (p_2)=0$. This conclude the proof of Proposition \ref{1-3}.
\end{proof}

\begin{proof}[Proof of Theorem \ref{mainthm3}] We need only apply Proposition \ref{1-3} with $\sigma=|F(z)|^2$, which is strongly plurisubharmonic as  $F$ is a local biholomorphism from $M$ into $\BB^n$.
\end{proof}
\medskip

We next give an immediate application of Theorem \ref{mainthm3}.

\begin{theorem}
Let $M$ be a Stein manifold.
Suppose its Bergman metric is well-defined near  point $p\in M$
and its holomorphic sectional curvature is a negative constant near $p$.
Then there exists a complex analytic hypersurface
$E^{*}\subset M$
such that
$ 
M\setminus E^{*}
$
is biholomorphic to $\BB^{n}\sm E$,
where $E^*$ could be empty and $E$  is a pluripolar closed subset of $\BB^{n}$.
\end{theorem}

\begin{proof}
By Proposition \ref{99-01},
we have a complex hypersurface
$E^{*},$ 
which may be empty,
such that the Bergman metric is well-defined on
$M\setminus E^{*} $.
By the analyticity of the Bergman metric and the assumption,
the Bergman metric on
$M\setminus E^{*}$ 
has constant holomorphic sectional curvature.
Since $M\setminus E^{*}$ is also Stein by a result of Docquier and Grauert \cite{DG},
the rest of the proof follows from
Theorem \ref{mainthm3}.
\end{proof}

We notice that the above theorem is optimal
by the following example.

Let
\[
\Omega
=
\left\{
(z,w)\in\mathbb C^{2}
:
|z|^{2}< (1+|w|)^{-2}
\right\}.
\]

Then $\Omega$ is an unbounded strongly pseudoconvex manifold.

Let
\[
E^{*}
=
\{(z,0)\in\mathbb C^{2}: z\in\mathbb C\},
\]
and denote
$E=E^{*}.$ 
Then
\[
\Omega\setminus E
\longrightarrow
{\BB^2} \setminus E
\]
is biholomorphic via
\[
F=(z,zw).
\]


The Bergman form on the  unit 2-ball is given by 
\[
\omega^B_{{\BB}^{2}}
=
3i\partial\bar\partial
\log\!\left(\frac{1}{1-|Z|^{2}-|W|^{2}}\right).
\]
Hence 
\[
\omega_{\Omega\sm E^*}^B=F^{*}(\omega^B_{{\BB^{2}}\sm E})
=
-3i\partial\bar\partial
\log
\!\left(
1-|z|^{2}(1+|w|^{2})
\right),
\]
which is the Bergman form of $\Omega$.

Near  $E^*$,
\[
F^{*}(\omega_{B^{2}})
=
3i\partial\bar\partial
\left(
|z|^{2}(1+|w|^{2})
+O(|z|^{4})
\right).
\]
Therefore
\[
\omega_{\Omega}^B=
3i
\left\{
(1+|w|^{2})
\,dz\wedge d\bar z
+
\bar z w dz\wedge d\bar w
+
z\bar w\,dw\wedge d\bar z
+
|z|^{2}\,dw\wedge d\bar w
+O(|z|^2)
\right\}.
\]
Hence, \(\omega^B_{\Omega}\) is degenerate along \(E^*\).

Now, for any closed subset \(E_0 \subset E^*\), by the removable singularity theorem for $L^2$-integrable holomorphic functions, there exists a biholomorphic map from \(\Omega \setminus E_0\) to a subdomain of \(\mathbb{B}^2\) if and only if \(E_0 \equiv E^*\), since this would imply that   the Bergman form of \(\Omega \setminus E_0\) is  positive definite.

\begin{remark}

More generally, let $M$ be a Stein manifold whose Bergman metric is well defined, and let $D$ be a bounded pseudoconvex domain with complete Bergman metric. In fact, $D$ may more generally be taken to be a bounded domain in a Stein manifold equipped with a complete Bergman metric.

Suppose that there exists a germ of a holomorphic map $f$ at $p_0\in M$ such that
\[
f(p_0)\in D
\]
and
\[
f^{*}(\omega_D^B)=\lambda\,\omega_M^B
\]
near $p_0$, for some constant $\lambda>0$. By a theorem of Mok \cite{Mo2}, $f$ extends to a holomorphic map
\[
F\colon M\longrightarrow D
\]
which is likewise a local isometry up to the constant factor $\lambda$.

Although Mok states his theorem for bounded domains, the interior continuation and
monodromy arguments in his proof require only that the canonical Bergman-Bochner map
${\mathcal B}_D$ of $D$ be an injective holomorphic isometric immersion and that
$(D,\omega_D^B)$ be complete. Thus, Mok's argument applies equally well in the
setting considered here. For the reader's convenience, we provide below a 
sketch of Mok's argument, and refer to Mok's paper
\cite{Mo2} for  details.

For simplicity, we assume $D$ is a domain and  Let $(U,z)$ be a holomorphic chart near $p_0$
such that $f$ is
biholomorphic on $U$, with
\[
f^*(\omega_D^B)=\lambda\,\omega_M^B.
\]
Then $\mathcal B_D\circ f$ is an isometric embedding of
$(U,\lambda\,\omega_M^B)$ into $(\mathbb P^\infty,\omega_{\mathrm{st}})$.
Since $\lambda\,\omega_M^B$ is again a real-analytic K\"ahler metric, the Calabi
extension theorem \cite{Ca}, in the form formulated in \cite{HuLi2}, shows 
that $f$ can be analytically continued along any curve issuing from $U$ as a
multivalued local holomorphic isometric map from $M$ into $D$.

We next recall arguments of Mok \cite{Mo2} as follows: For $p,q$ sufficiently
close to $p_0$, one has from  $
f^*(\omega_D^B)=\lambda\,\omega_M^B$ the following:
\[
\left(
\frac{
K_M(p,p)\,K_M(q,q)
}{
K_M(p,q)\,K_M(q,p)
}
\right)^{\lambda}
=
\frac{
K_D\bigl(f(p),f(p)\bigr)\,
K_D\bigl(f(q),f(q)\bigr)
}{
K_D\bigl(f(p),f(q)\bigr)\,
K_D\bigl(f(q),f(p)\bigr)
}.
\]
The expression inside the parentheses on the left-hand side is 
 the Calabi diastasis function of $M$, and is  real analytic wherever
it is defined and takes  value $1$ at $p=q=p_0$. In particular, for $p,q\approx p_0$, it holds that 
\[
\left|
\frac{
K_M(p,p)\,K_M(q,q)
}{
K_M(p,q)\,K_M(q,p)
}
\right|^{\lambda}
=
\left|
\frac{
K_D\bigl(f(p),f(p)\bigr)\,
K_D\bigl(f(q),f(q)\bigr)
}{
K_D\bigl(f(p),f(q)\bigr)\,
K_D\bigl(f(q),f(p)\bigr)
}
\right|.
\]
For each fixed $q$, the left-hand side is a well-defined real-analytic function
away from the zero set of $K_M(\cdot,q)$. Hence, by uniqueness of real-analytic functions, one may continue the above identity along curves in the $p$-variable.
Setting $q=p$, one obtains as in Mok \cite{Mo2} for any new branch
$\widehat f$ of $f$ near $p_0$,
\[
\bigl|K_D\bigl(\widehat f(p),\widehat f(p)\bigr)\bigr|\,
\bigl|K_D\bigl(f(p),f(p)\bigr)\bigr|
=
\bigl|K_D\bigl(\widehat f(p),f(p)\bigr)\bigr|^2.
\]
By Mok's trick using the equality case of the Schwarz inequality, this gives
\[
\widehat f(p)=f(p),
\]
since $A^2(D)$ separates points. Thus the analytic continuation has trivial
monodromy, and $f$ extends to a globally defined local  biholomorphic map
\[
F\colon M\longrightarrow D.
\]
Note that this conclusion does not require $M$ to be Stein.

In particular, since $D$ is bounded, $M$ admits a smooth bounded strictly
plurisubharmonic function
\[
\sigma:=|F|^2.
\]
Proposition~\ref{1-3} therefore implies that the Bergman space $A^2(M)$
separates points of $M$.

We thank Yuan Yuan for helpful discussions that led to this observation and for suggesting that we include it in this paper for future reference.
\end{remark}

\section{Flat Bergman metrics}
We prove in this section the following:

\begin{proposition}\label{thm2.1}
Let $M$ be a Stein manifold of complex dimension $n\ge 1$. Assume its Bergman metric of $M$ is well-defined  and has constant zero holomorphic sectional curvature.  Then its Bergman space must separate points on $M$.
\end{proposition}

This together with the result  by the authors and Treuer proved in  \cite{HuLi1} gives the following:

\begin{theorem}\label{mainthm4}
Let $M$ be a Stein manifold of complex dimension $n\ge 1$. Assume the Bergman metric of $M$ is well-defined. 
Then it can not be flat, i.e.,
it can not   have  constant  zero holomorphic sectional
  curvature on $M$.
\end{theorem}


\begin{proof}[Proof of Proposition \ref{thm2.1}]
Let $M$ be a Stein manifold.
Assume that the Bergman metric  of $M$ exists and has constant zero holomorphic sectional curvature.
Let $p_0\in M$ and let 
$(U,z)$ with $z(p_0)=0$ be a holomorphic coordinate chart. (Still
write $U$ for $z(U)$ for brevity when there is no confusion).
Shrinking $U$ if necessary, there is a biholomorphic map  $F:=(f_1,\cdots,f_n)$ from $U$
to a neighborhood $V\subset \CC^n$ of $0$ such that $F(p_0)=0$ and
$F^*(\o_{\eucl})= \o_M^B,$ where $\o_{eucl}=i\p\ov{\p}|z|^2=i\p\ov{\p}\log(e^{|z|^2})$ is the
Euclidean form of $\CC^n$ and   $d^2s_{\eucl}=2\h{Re}\left(\sum_{j=1}^{n}dz_j\otimes \ov{dz_j}\right).$

 Notice that there  is  a holomorphic isometric embedding  $\cal P$ from 
 $(\CC^n, \o_{euclid})$ into $(\PP^\infty, \o_{st}^\infty)$. Indeed, write $e^{|z|^2}=1+\sum_{j=1}^\infty \frac{|z|^{2j}}{j!}=1+\sum_{k=1}^n|z_k|^2+ \sum_{\ell=2}^\infty |P_\ell|^2$ with $P_\ell$ 
 certain holomorphic monomials  of degree at least 2.
  We can then define
 ${\cal P}=[ 1,z_1,\cdots,z_n, P_2, \cdots,P_\ell,\cdots]$.
By the Calabi rigidity theorem \cite{Ca} 
there is    a one-to-one linear isometric embedding  $\cal L$ from   $X_{\cal P}\subset (\PP^\infty, \o_{st}^\infty)$   into  $(\PP^\infty, \o_{st}^\infty)$ such  that
  $${\cal B}={\cal L}\circ {\cal P}\circ F,$$
  where $\B$ is a Bergman-Bochner map of $M$.
 Since ${\cal B}, {\cal L}$ and  ${\cal P}$ are one-to-one,   by the Calabi extension theorem \cite{Ca, HuLi2},  one also immediately concludes that $F$ extends to a local isometric biholomorphism from $M$ onto a certain domain  denoted by $D$ in $\CC^n$.  
 
Next notice that   $F^*(i\p\ov{\p}\log(e^{|z|^2}))=i\p\ov{\p} |F(p)|^2.$ 
Let  $\{\phi_j\}_{j=0}^N$ be an orthonormal basis of the
Bergman space $A^2(M)$ with $\phi_0(p_0)\not =0$ and $\phi_j(p_0)=0$ for
$j\ge 1$. Write $\phi_j=\ld_j \phi_0$ near $p_0$ for $j\ge 1$. Then
$K_M(p,p)=\phi_0(p)\wedge \ov{\phi_0(p)}(1+\sum_{j=1}^{N}|\ld_j|^2)$.
From the local isometric property of $F$,
we have near $p_0$
$$i\p\ov{\p}\log(1+\sum_{j=1}^{N}|\ld_j|^2)=i\p\ov\p \log(e^{|F(z)|^2}).$$
Since $\ld_j(p_0)=0$ and $F(p_0)=0$,   which  are both holomorphic, we conclude that near $p_0$
$$\sum_{j=1}^{N}|\ld_j|^2=\sum_{j=1}^{\infty}\frac{|F|^{2j}}{j!}.$$

In the local coordinates mentioned above, we have  $F(z)=z$ and thus $\{F^\a\}_{\a}$ is a linear independent set. Recall  also the binormial expansion $(x_1+\cdots +x_n)^{k}=\sum_{|\a|=k}\frac {k!}{\a!}x^\a$ which can be easily seen by applying $\p_x^\a$ of both sides
with $|\a|=k$. 
Write $\L= (1,\cdots, \ld_j(p),\cdots)$. Then
we have $$\|\L\|^2=\sum_{j=1}^{N}|\ld_j|^2=\sum_{|\b|\ge 1}^{\infty}\big|\frac{F^\b}{\sqrt{\b!}}\big|^2.$$
 Making use of  the Calabi rigidity theorem or the D'Angelo lemma \cite{DA},
there is  a linear one-to-one  and surjective isometry $L$ from   $X_{F_0}\subset \ell^2$ to $X_\L\subset \ell^2$ 
such that 
$$(1,\cdots, \ld_j(p),\cdots)=L\left((1,f_1(p),\cdots, f_n(p),\cdots, \frac{F^{\b}(p)}{\sqrt{\b!}},\cdots)\right), \ \h{ and }$$
$$\left(1,f_1(p),\cdots, f_n(p),\cdots, \frac{F^{\b}(p)}{\sqrt{\b!}},\cdots\right)=L^{-1}\left((1,\cdots, \ld_j(p),\cdots)\right).$$
Here $F_0=(1,f_1(p),\cdots, f_n(p),\cdots, \frac{F^{\b}(p)}{\sqrt{\b!}},\cdots)$.
Hence we conclude that $$\{\phi_0, \phi_0 f_1,\cdots, \phi_0 f_n,\cdots, \phi_0 \frac{F^{\b}(p)}{\sqrt{\b!}},\cdots\}$$ is also an orthonormal basis of $A^2(M)$. In particular, since $A^2(M)$ is base point free, we see that $\phi_0\not = 0$

 We equip $M$ with its Bergman-K\"ahler metric $\o^B_M=F^*(\o_{eucl})$, which is flat.  For $p_0\in M$, as mentioned above, we also use $z=F(p)$ as a local coordinate  for $p\approx p_0$ . Then as in the previous discussion,
we have  for any $\a, \b\in A^2(M)$, $$(\a,\b)=\int_{M}<\a,\b> d{\rm{vol}}_M= b_n\int_{M}\a\wedge \ov{\b}\ \ \h{with  } b_n={{( i)}^n}(-1)^{\frac{n^2-n}{2}},$$
where $d \rm{vol}_M= \frac{\o^n_B}{n!}$ is the volume form induced by the metric $\o_{\eucl}=i\p\ov\p(|z|^2)$.  We will write $|\a|^2=<\a,\a>$ for any $\a\in A^2(M).$
Then the orthonormality of $$\{\phi_0, \phi_0 f_1,\cdots, \phi_0 f_n,\cdots, \phi_0 \frac{F^{\b}(p)}{\sqrt{\b!}},\cdots\}$$  gives the following:
$$b_n\int_M F^\a\ov{F^\b} \phi_0\wedge \ov  \phi_0= \d_{\a}^{\b}\a!.$$


Let $D=F(M)$ and let $\nu$ be the pull back measure on $\CC^n$ defined by 
$$
\nu(E) =\int_{F^{-1}(E\cap D)} <\phi_0(q),{ \phi_0(q)}>d{\rm{vol}}_M(q)=b_n\int_{F^{-1}(E\cap D)} \phi_0(q)\wedge \ov{ \phi_0(q)}
$$
for any Lebesgue measurable set $E\subset\CC^n$. 
Then
$$
L^{\nu}[z^\alpha \zbar^\beta]:=\int_{\CC^n} z^\alpha\zbar^\beta d \nu= b_n\int_M \overline{F(q)}^\beta F(q)^\alpha \phi_0(q)\wedge \ov{ \phi_0(q)}=\alpha! \delta_{\alpha \beta}.
$$
On the other hand,  by a direct  computation (e.g., see  (5.19)  of \cite{HuLi1}), letting $dv$ be the Lebesque measure and 
$$
d\mu=\pi^{-n} e^{-|z|^2} dv,
$$
then
$$
L^{\mu}[z^\alpha \zbar^\beta]:=\int_{\CC^n} z^\alpha \zbar^\beta d\mu=\alpha! \delta_{\alpha \beta},
$$
where $dv$ is the Lebesque measure on $\CC^n$.
By  the classical uniqueness  of the moment problem (see, for example, Theorem 5.1 in \cite{HuLi1} or \cite{Sch}), one has
$$
d\nu=d\mu.
$$
This, in particular,  implies that $D=\CC^n\setminus E$ for a certain $E$ with $m(E)=0$. 

For any $z_0\in D$,   $F^{-1}(z_0)$ has no accumulation point in $M$. Let $F^{-1}(z_0)=\{q_1,\cdots, q_N\}$ with $N$ being  finite or $\infty$. 
Then for any given  $q_j$,  let $\delta_j>0$ be sufficiently small such that $F$ is biholomorphic from a neighborhood $U(q_j, \delta_j)$ of $q_j$, which is contained in a coordinate chart defined above near $q_j$, to $B(z_0, \delta_j)$ with inverse $G_j(z)$.  
Notice that in the holomorphic charts we are using, $F$  is simply the identity map. Therefore we have 
\begin{eqnarray*}
\int_{F^{-1}\left(B(z_0, \delta_j)\right)} |\phi_0(q)|^2dvol_M(q)=
\int_{B(z_0, \delta_j) }  d\nu(z)
=\int_{B(z_0, \delta_j) } \pi^{-n} e^{-|z|^2} dv(z)
\end{eqnarray*}
and
\begin{eqnarray*}
 \int_{F^{-1}\left(B(z_0, \delta_j)\right)}  |\phi_0(q)|^2dvol_M(q)
&\ge& \int_{U(q_j,\d_j)}  |\phi_0(q)|^2dvol_M(q)\\ 
&=&\int_{B(z_0, \delta_j)} |\phi_0(G_j(z))|^2 |G_j^*\left(dvol_M(q)\right)\\
&=& 2^n\int_{B(z_0, \delta_j)} |\phi_0(G_j(z))|^2 dv(z).
\end{eqnarray*}
 Here we notice that the last equality holds as $G_j$ is a holomorphic  isometry. And by our convention here, $d vol_M= 2^n dv$.
Since $\d_j$ could be made arbitrarily small, it  thus follows that
\begin{equation} \label{IIE}
\pi^{-n}e^{-|z_0|^2}\ge 2^n|\phi_0(q_j)|^2,\ \h{or  } e^{-|F(q_j)|^2}  \ge 2^n\pi^n  |\phi_0(q_j)|^2.
\end{equation}
By the arbitrariness of $q_j$, we have for any $q\in M$ the following: 
\begin{equation}
2^n\pi ^ne^{|F(q)|^2} \big| \phi_0(q) \big|^2
\le 1 
\end{equation}
and thus 
\begin{equation}   
\sigma(z):= |F(q)|^2+\log |\phi_0(q)|^2\le -n\log \pi-n\log 2.
\end{equation}
Notice that $|F(q)|^2$ is strongly plurisubharmonic and  $\log |\phi_0(q)|^2$ is pluriharmonic.  Thus $\sigma$ is a smooth strongly plurisubharmonic function over $M$ bounded from above.
By Proposition \ref{1-3}, we conclude
the proof of Proposition \ref{thm2.1}.
\end{proof}
\begin{proof}[Proof of Theorem \ref{mainthm1}] It follows immediately from Theorems \ref{mainthm3} and \ref{mainthm4}.
\end{proof}
Note that the complement of a complex analytic hypersurface in a
Stein manifold is again Stein by \cite{DG}. By the real analyticity of the Bergman
metric, together with Proposition~\ref{99-01} and Theorem~\ref{mainthm4},
we obtain the following  theorem:

\begin{theorem}
Let $M$ be a Stein manifold, and let $p\in M$.
Assume that the Bergman metric of $M$ is well defined in a neighborhood of $p$.
Then the Bergman metric cannot be flat in any neighborhood of $p$.
Equivalently, it cannot have identically vanishing holomorphic sectional
curvature in a neighborhood of $p$.
\end{theorem}

.

\section{Bergman metrics with positive constant holomorphic sectional curvature}
One might expect that a similar phenomenon, as in Theorem \ref{mainthm1}, holds in the case of constant positive sectional curvature. However, our examples in this section, based on the theory of  hyper-elliptic algebraic manifolds \cite{GH}, show that the situation is quite different even in dimension one. We next restate Proposition \ref{333} as follows:



\begin{proposition}\label{444}
For every \(n \in \mathbb{N}\), there exists a Stein manifold \(M\) of complex dimension \(n\), embedded in ${\mathbb C}^{n+1}$,  such that its Bergman space does not separate points in \(M\), yet the holomorphic sectional curvature of its Bergman metric is a positive constant. Moreover, there exist  Stein manifolds whose Bergman metric is well defined but whose Bergman space does not separate points, and whose scalar curvature can  be a positive, negative, or zero constant.
\end{proposition}

We first describe how to construct such an example in the case of complex dimension one. 
Let $X$ be the  normalization of the   compactification in \(\mathbb{P}^2\) of the affine smooth curve $X_0$ in \(\mathbb{C}^2\) with coordinates \((x,y)\) defined by
\[
y^2=P(x)
:= 1+x^6.
\]
Then $X$ is an example of a compact hype-elliptic Riemann surface.  Note that  the holomorphic extension of the projection map $\pi$, which sends $(x,y)\in X_0$ to $x\in {\mathbb P}^1$, is a two-to-one branched holomorphic covering with exactly six branch points at  the roots of $x^6+1=0$.
Since the points at infinity  do not contribute ramification, by the Riemann--Hurwitz genus formula:
$$
2g(X)-2=2(2g({\mathbb P}^1)-2)+6,
$$
the geometric genus of $X$ is exactly $g=2$.

 Since \(X\) is compact,
\[
A^2(X)=H^0(X,K_X),
\]
where \(K_X\) denotes the canonical bundle. By the Riemann-Roch theorem, $\hbox{dim}_{\mathbb C}A^2(X)=g=2$.

Choose an orthonormal basis
\[
\{\omega_1,\omega_2\}
\]
of \(H^0(X,K_X)\). The Bergman kernel on the diagonal is
\[
K(z,\bar z)=i\sum_{j=1}^2 \omega_j(z)\wedge \overline{\omega_j(z)}
= i\,k_X(z,\bar z)\,dz\wedge d\bar z
\]
where in a local holomorphic chart \(z\), $\omega_j=f_j(z) dz$, so $k_X(z)=\sum_{j=1}^2 |f_j(z)|^2$. Hence, the Bergman metric is
\[
\omega =i\partial\bar\partial \log k_X(z,\bar z).
\]



Now let $X_0'=X_0\sm E_0$ with $E_0:= \{(x,y)=(a,0)\in {\mathbb C}^2: a^6=-1\}$.
By the removable singularity theorem for \(L^2\)-integrable holomorphic functions across a subset $E$ of isolated points,  the Bergman space of \(A^2(X\setminus E_0)\) can be identified with \(A^2(X)\). 
In particular,  
 the Bergman space $A^2(X_0 \setminus E_0)$ has complex dimension two.

Indeed, consider the holomorphic one-forms on $X_0$:
\[
\eta_1=\frac{dx}{y},
\qquad
\eta_2=\frac{x\,dx}{y}.
\]
Notice the inner product  for any $\omega, \eta\in A^2(X_0)$ is given by 
\[
( \omega,\eta)
=
i\int_X\omega\wedge\overline{\eta},
\]
One  verifies easily  that $\eta_1,\eta_2\in A^2(X_0)$. (We will verify this in greater detail later even in higher dimensions.). Write
\[
A=\|\eta_1\|^2>0,
\qquad
B=\|\eta_2\|^2>0.
\]

Let $\sigma(x,y)=(-x,y)$. Notice that
$\sigma^*(\eta_1\wedge \overline{\eta_2})
=-\eta_1\wedge \overline{\eta_2}$.
 We obtain that $$ (\eta_1,\eta_2)=
i\int_X\eta_1\wedge\overline{\eta_2}=i\int_X
\sigma^*(\eta_1\wedge\overline{\eta_2})=- (\eta_1,\eta_2).$$
Thus $(\eta_1,\eta_2)=0$ and an orthonormal basis of $A^2(X_0)$ is given by 
\[
\omega_1
=
\frac{1}{\sqrt A}\frac{dx}{y},
\qquad
\omega_2
=
\frac{1}{\sqrt B}
\frac{ xdx}{y}.
\]

Notice that the canonical embedding $\mathcal{B} := [\omega_1, \omega_2] = [1, \sqrt{\tfrac{A}{B}}\,x]$ from $X_0$ into $\mathbb{P}^1$ takes the same value at the points $(x,y)$ and $(x,-y)$. Hence, the Bergman space $A^2(X_0 \setminus E_0)$ does not separate points in $X_0 \setminus E_0$.
However, the Bergman metric on $X_0 \setminus E_0$ is the pullback of the Fubini--Study metric on $\mathbb{P}^1$, and is  locally isometric to it. In particular, this defines a well-defined Bergman metric on the Stein manifold $X_0'=X_0 \setminus E_0$ of complex dimension one, with constant Gaussian curvature $+2$. Nevertheless, its Bergman space fails to separate points.

Next, let $X_{0,m}'$ be the $m$-fold product of $X_0'$ with $m \in \mathbb{N}$. We define
\[
M := X_{0,m}' \times \mathbb{B}^{n} \subset \mathbb{C}^{2n+m}.
\]
Then $M$ is a Stein manifold of complex dimension $n+m$. Moreover,
\[
A^2(M)=A^2(X_{0,m}'){\otimes} A^2(\mathbb{B}^{n}).
\]
Therefore, the Bergman metric $\omega_M$ is well defined and has constant scalar curvature $S=2m-n$. Nevertheless, $A^2(M)$ does not separate points of $M$. Notice that $S$ can be positive, negative, or zero.

We now proceed to the higher-dimensional construction.

\noindent 

For $n\in \NN$, define
$$
X_0=\Big\{(z, w)\in \mathbb{C}^n\times \mathbb{C}: w^2=1+\sum_{j=1}^n z_j^{2n+4}\Big\}
$$
and $$
E=\{z\in \mathbb{C}^n:1+\sum_{j=1}^n z_j^{2n+4}=0\}
$$
Thus 
$$
\pi: X_0\setminus\{w=0\}\to \mathbb{C}^n \setminus E,\quad (z, w)\to z
$$
is an unbranched two-sheeted covering. Both $X_0 \setminus E$ and  $X_0$ are  Stein.
 Since  the holomorphic function space over $X_0$ is isomorphic to
$$
{\mathcal O}(X_0)={\mathcal O}(\mathbb{C}^{n+1})/\left(w^2-(1+\sum_{j=1}^n z_1^{2n+4})\right)
$$
as $X_0$ is a closed complex submanifold of $\mathbb{C}^{n+1}$,
 every holomorphic function on $X_0$ can be uniquely written as
$$
f=a(z_1, z_2)+b(z_1, z_2) w
$$
where $a, b$ are entire  holomorphic functions  on $\mathbb{C}^n$.

Recall the Bergman space is given by  square-integrable holomorphic $(n,0)$-forms
with inner product
\[
(\alpha,\beta)
 =b_n \int_{X_0} \alpha\wedge\overline{\beta}
\]
where $b_n=(i)^n (-1)^{\frac{n^2-n}{2}}=i^{n^2}$.

We now define $\omega_0={d z_1\wedge d z_2\wedge\cdots \wedge dz_n\over w}$. We claim that  $\omega_0$ is a well-defined holomorphic $n$-form with no zero at any point. This is obvious away from $E$. For $z_0=(z_{1,0},\cdots,z_{n,0})\in E$, assume without loss of generality that $z_{1,0}\not=0$. We can use $(z_2,\cdots, z_n, w)\approx (z_{2,0},\cdots z_{n, 0},0)$ for   holomorphic coordinates near $(z_{2,0},\cdots, z_{n, 0},0)$. Since along $X_0$, 
$$2wdw=(2n+4)\sum_{j=1}^n z_j^{2n+3}dz_1,$$
we see that $\omega_0= \frac{1}{(2n+4) z_1^{2n+3}} dw\wedge dz_2\wedge\cdots \wedge dz_n$ near $ (z_0,0)$, which is holomorphic and non-zero near $ (z_0,0)$.

Therefore any holomorphic $2$-form  $\omega$ on $X_0$ can be written as
$$
\omega=f \omega_0=(a(z)+b(z)w) {dz_1\wedge \cdots \wedge dz_n\over w}=({a(z)\over w}+b(z) ){dz_1\wedge \cdots \wedge dz_n},\quad \omega_0={d z_1\wedge\cdots\wedge d z_n\over w},
$$
with $a(z), b(z)$ entire on $\mathbb{C}^n$.
Let $U_1=\mathbb{C}\setminus [0, \infty)$ and $U_2=\mathbb{C}\setminus (-\infty, 0]$ and $\lambda_j$ be holomorphic functions on $U_j$ such that $\lambda_j(s)^2=s$ for any $s\in U_j$. Let
$$
\gamma_j(z)=\lambda_j(1+\sum_{j=1}^n z_j^{2n+8}), \quad \hbox{when}\ 1+\sum_{j=1}^n z_j^{2n+4}\in U_j.
$$

 Then $w^2=P(z)$ has two solutions,
\[
w=\gamma_j(z) \quad \text{and} \quad w=-\gamma_j(z), \qquad \hbox{for  }1+\sum_{j=1}^n z_j^{2n+4}\in U_j.
\]
Let
\[
W_1=\left\{z\in \mathbb{C}^n\setminus E:\ 1+\sum_{j=1}^n z_j^{2n+4}\in U_1\right\}, \qquad
W_2=\left\{z\in \mathbb{C}^n\setminus E:\ 1+\sum_{j=1}^n z_j^{2n+4}\in U_2\right\}.
\]
Let $\chi_1$ and $\chi_2$ be a partition of unity subordinate to $\{W_1, W_2\}$. Then

\begin{eqnarray*}
(\omega, \omega)&=&
\sum_{j=1}^2 \int_{ \mathbb{C}^n\setminus E}\chi_j \left(\left |{a(z)\over \gamma_j(z)}+b(z)\Big|^2  +
\right |{a(z)\over -\gamma_j(z)}+b(z)\Big|^2 \right )dV(z)\\
&=&
2\int_{\mathbb{C}^n\setminus E} \left({|a(z)|^2\over |1+\sum_{j=1}^n z_j^{2n+4}|}+|b(z)|^2 \right)dV(z),
\end{eqnarray*}
where $dV=(i)^{n^2} d z_1\wedge d z_2\wedge\cdots \wedge dz_n \wedge d\overline{z_1}\wedge d \overline{z_2}\wedge\cdots \wedge d\overline{z_n}. $

Notice that $\int_{\mathbb{C}^n}|b(z)|^2 dV=\int_{\mathbb{C}^n\setminus E}|b|^2 dv<\infty$ implies that $b(z)\equiv 0$.  If $\omega=(a+bw)\omega_0\in A^2(X_0)$, then $b\equiv 0$ and
$$
\infty>\int_{\mathbb{C}^n\setminus E} {|a(z)|^2 \over |1+\sum_{j=1}^n z_j^{2n+4}|} dV(z)\ge
\int_{\mathbb{C}^n } {|a(z)|^2 \over 1+|z|^{2n+4}} dV(z).
$$

Write $z_j=r_je^{i\theta_j}$. Then the Taylor expansion of $a$ is
\[
a(z)=\sum_{\alpha=(\alpha_1,\cdots,\alpha_n)\ge 0} c_\alpha z^\alpha
=\sum_{\alpha\ge 0} c_\alpha r_1^{\alpha_1}\cdots r_n^{\alpha_n}
e^{i(\alpha_1\theta_1+\cdots+\alpha_n\theta_n)}.
\]

Then 
$$\int_{\mathbb{C}^n } {|a(z)|^2 \over |1+|z|^{2n+4}} dV(z)=\sum_{\alpha\ge 0}\int_{r_1,\cdots,r_n\ge 0}\frac{|c_\alpha|^2r_1^{2\alpha_1+1}\cdots r_n^{2\alpha_n+1}}{1+\left(r_1^2+\cdots r_n^2\right)^{n+2}} dr_1\cdots dr_n<\infty.
$$
This implies that $a(z)$ must be a polynomial of degree no larger than $k$ with 
$$
2k+n+(n-1)-(2n+4) <-1 \Rightarrow k<2.
$$
Therefore, if $\omega=(a+bw) \omega_0\in A^2(X_0)$ then $b=0$ and 
$$
a(z)=c_0+\sum_{j=1}^nc_j z_j,\ \  c_j\in \mathbb C.
$$
Next, we prove the following lemma, which may already appear in the literature. Since we were unable to find a reference, we provide a detailed proof.

\begin{lemma}\label{222}
Let
\[
P(z)=1+\sum_{j=1}^{n} z_j^{2n+4},
\qquad
E=\{z\in\mathbb C^n:P(z)=0\}.
\]
Then, for every \(1\leq \ell\leq n\),
\[
\int_{\mathbb C^n\setminus E}
\frac{|z_\ell|^{2k}}{|P(z)|}\,dV(z)<\infty, \ \ \hbox{where  }k=0,1.
\]
\end{lemma}

\begin{proof}
For simpliicty, we only verify the case $k=1$.
By symmetry, it suffices to consider \(\ell=1\). We divide the proof
into local integrability near \(E\) and integrability at infinity.

We showed that \(E\) is a smooth complex hypersurface as $dP(z)\not = 0$ along $E$. Therefore, near each point of \(E\), the holomorphic implicit function
theorem provides local holomorphic coordinates
\[
u_1=P(z),\quad u_2,\ldots,u_n.
\]
Since \(|z_1|^2\) is locally bounded, the integrand is bounded by a
constant multiple of \(1/|u_1|\). Moreover,
\[
\int_{|u_1|<\varepsilon}\frac{dA(u_1)}{|u_1|}
=
2\pi\int_0^\varepsilon dr
<\infty.
\]
Here $dA(u_1)$ is the Lebesque measure of ${\mathbb C}$.
Thus
\[
\frac{|z_1|^2}{|P(z)|}
\]
is locally integrable on \(\mathbb C^n\). In particular, its integral
over every bounded subset of \(\mathbb C^n\setminus E\) is finite.

It remains to study the behavior at infinity. Set
\[
m=2n+4,
\qquad
Q(z)=\sum_{j=1}^{n}z_j^m.
\]
For \(R\geq 1\), let
\[
A_R=\{z\in\mathbb C^n:R<|z|<2R\}.
\]
Using the change of variables \(z=R\zeta\), we obtain
\[
dV(z)=R^{2n}dV(\zeta),
\qquad
|z_1|^2=R^2|\zeta_1|^2,
\]
and
\[
P(R\zeta)
=
1+R^mQ(\zeta)
=
R^m\bigl(R^{-m}+Q(\zeta)\bigr).
\]
Consequently,
\begin{align*}
\int_{A_R}\frac{|z_1|^2}{|P(z)|}\,dV(z)
&=
R^{2n+2-m}
\int_{1<|\zeta|<2}
\frac{|\zeta_1|^2}
{|Q(\zeta)+R^{-m}|}\,dV(\zeta) \\
&=
R^{-2}
\int_{1<|\zeta|<2}
\frac{|\zeta_1|^2}
{|Q(\zeta)+R^{-m}|}\,dV(\zeta).
\end{align*}
We claim that there exists a constant \(C>0\), independent of
\(0\leq\varepsilon<<1\), such that
\[
\int_{1<|\zeta|<2}
\frac{dV(\zeta)}{|Q(\zeta)+\varepsilon|}
\leq C.
\]
Indeed, the only critical point of \(Q\) is the origin, because
\[
\frac{\partial Q}{\partial \zeta_j}
=
m\zeta_j^{m-1}.
\]
Hence \(dQ\neq 0\) on the compact annulus
\(\{1\leq|\zeta|\leq2\}\). Near every level hypersurface
\(\{Q+\varepsilon=0\}\), one may use \(Q+\varepsilon\) as one
holomorphic coordinate. The local integral is therefore reduced to
\[
\int_{|u|<\delta}\frac{dA(u)}{|u|}<\infty.
\]
By compactness, these local estimates can be chosen uniformly in
\(\varepsilon<<1\).

Since \(|\zeta_1|\leq2\) on the annulus, it follows that
\[
\int_{A_R}\frac{|z_1|^2}{|P(z)|}\,dV(z)
\leq \frac{C}{R^2}.
\]
Taking \(R=2^k\) and summing over dyadic annuli gives
\begin{align*}
\int_{\{|z|>1\}\setminus E}
\frac{|z_1|^2}{|P(z)|}\,dV(z)
&\leq
C\sum_{k=0}^{\infty}2^{-2k}
<\infty.
\end{align*}
Combining this estimate with local integrability on bounded subsets
proves that
\[
\int_{\mathbb C^n\setminus E}
\frac{|z_1|^2}{|P(z)|}\,dV(z)<\infty.
\]
\end{proof}

Lemma \ref{222} shows that $\{\omega_0,z_1\omega_0, \cdots, z_n\omega_0\}$ forms a basis of $A^2(X_0)$.
Write $\eta_0=\omega_0$ and $\eta_j=z_j\omega_0$ for $n\le j\ge 1$.

 Let
$$
A=(\eta_0, \eta_0),\quad B=(\eta_1,\eta_1)=\cdots=(\eta_n,\eta_n)
$$ which 
are finite and positive. 
We claim that $\{{1\over \sqrt{A}} \eta_0, {1\over \sqrt{B}} \eta_1,\cdots, {1\over \sqrt{B}}\eta_n\}$ is an  orthonormal basis for $A^2(X_0)$. This will follow from the fact $(\eta_j, \eta_k)=0$ if $j\ne k$. 
To prove the last fact, we consider complex numbers $\tau_1,\cdots, \tau_n$ satisfying:  $\tau_j^{2n+4}=1$,  we define
$$
\Phi_{\tau}(z, w)=(\tau_1 z_1, \cdots, \tau_n z_n, w),\quad \tau=(\tau_1,\cdots, \tau_n).
$$
Then $\Phi_{\tau}: X_0\to X_0$ is a biholomorphic automorphism and
$$
(\Phi_{\tau}^*\alpha, \Phi_{\tau}^*\beta)=(\alpha, \beta)
$$
Since
$$
\Phi_{\tau}^*\eta_0=(\prod_{j=1}^n \tau_j) \eta_0, \quad \Phi_{\tau}^* \eta_j =\tau_j(\prod_{\ell=1}^n\tau_\ell)    \eta_\ell.
$$
Taking $\tau_j=(-1)^n\sqrt{-1}$ and $\tau_k=1$ if $k\ne j$, one has
$$
(\eta_0, \eta_j)
=(-1)^{n+1}\sqrt{-1} (\eta_0, \eta_j)
$$
This implies $(\eta_0, \eta_j)=0$. Similarly, by choosing a different $\tau_j$, one can prove $(\eta_j, \eta_k)=0$ if $j\ne k$. Therefore,
the Bergman space of $X_0\sm E$ is $(n+1)$ dimensional and has an orthonormal basis:
\[
\frac{1}{\sqrt A}\omega_0,\qquad
\frac{1}{\sqrt B}z_1 \omega_0,\cdots,\qquad
\frac{1}{\sqrt B} z_n \omega_0.
\]
 Consequently, since $\log |w|^2=\log |1+\sum_{j=1}^n z_j^{2n+4}|$ is pluriharmonic, we have
\[
\omega_B
 =i\partial\bar\partial\log(1+\frac{A}{B}\sum_{j=1}^n |z_j|^2)
\]
which is locally isometric to the Fubini-Study metric of ${\mathbb P}^n$
by the map $\sigma(
z,w)=[1, \sqrt{\frac{A}{B}}z].$ 
The holomorphic sectional curvature is therefore identical to $2$. Since the Bergman-Bochner map or the canonical map is given by $${\cal B}:X_0\sm E \rightarrow {\mathbb P}^n: (z,w)\to [1, \sqrt{\frac{A}{B}}z],$$ ${\cal B}$  is a two to one map as ${\cal B}(z,w)={\cal B}(z,-w)$. Thus, $A^2(X_0\sm E)$ does not separate the points in  the Stein manfold $X_0\sm E$.  We have completed the proof of the proposition.


\noindent Xiaojun Huang (huangx@math.rutgers.edu): Department of Mathematics, Rutgers Univer-
sity, New Brunswick, NJ 08903, USA

\noindent 
Song-Ying Li (sli@uci.edu): Department of Mathematics, University of California, Irvine,
Irvine, CA 92697, USA

\end{document}